\documentclass[12pt]{article}
\usepackage[margin=1in]{geometry}
\usepackage{graphicx}
\graphicspath{{Figures/}}
\usepackage[hidelinks]{hyperref}
\usepackage{subcaption}
\usepackage{float}
\usepackage{amssymb,amsmath,amsthm}
\usepackage{newtxtext,newtxmath}
\usepackage[dvipsnames]{xcolor}

\usepackage[backend=biber,style=numeric,maxbibnames=8]{biblatex}
\definecolor{addblue}{rgb}{0.1,0,0.8}

\usepackage[normalem]{ulem}

\definecolor{darkgrn}{rgb}{0,0.75,0}

\theoremstyle{plain} 
\newtheorem{theorem}{Theorem}[] 

\newtheorem{proposition}[theorem]{Proposition}
\newtheorem{corollary}[theorem]{Corollary}

\theoremstyle{definition} 
\newtheorem{definition}[theorem]{Definition}
\newtheorem{example}[theorem]{Example}

\theoremstyle{remark} 
\newtheorem*{remark}{Remark}

\newcommand{\R}{{\mathbb R}}

\newcommand{\B}{{\mathbb B}}
\newcommand{\C}{{\mathbb C}}

\title{\huge Minimal Homotopies in Three Dimensions:\\A Cable System Approach}
\author{Lia Buchbinder\thanks{lia.buchbinder@wsu.edu; corresponding author} \quad\quad Bala Krishnamoorthy \quad\quad Kevin R.~Vixie\\
Washington State University}
\date{} 

\begin{document}
\maketitle

\begin{abstract}
We study null homotopies of immersed spheres in $\R^3$ and the volume they sweep during contraction. 
For a smooth immersion with finitely many transverse self-intersections, we introduce a cable system that connects each bounded region of the complement to the exterior. 
From this construction we define the cable index and prove that it agrees with the Brouwer degree on each complementary region.
Using this identification, we derive a degree-weighted lower bound for the swept volume of any Lipschitz null homotopy. 
We show that the bound is attained whenever the homotopy is sense-preserving, meaning the surface moves in a consistent direction, and the index evolves monotonically along the homotopy.
In addition, in the case where the immersion arises as the boundary of an immersed ball, we construct an explicit homotopy that realizes this lower bound via a deformation of the ball.
Finally, we present a linear-time algorithm that computes all cable indices from a finite cable system, providing a concrete and computable method for evaluating the lower bound.
\end{abstract}

\section{Introduction}
We study null homotopies of immersed spheres in $\R^3$ and the volume they sweep during a contraction. 
A null homotopy of an immersed sphere $\iota(S^2)\subset \R^3$ is a continuous deformation that contracts the sphere to a point.  
We measure such a deformation by the \emph{swept volume} of the map $H:S^2\times[0,1]\to \R^3$, counted with multiplicity. 
Our goal is to estimate the smallest possible swept volume among all Lipschitz null homotopies.

This problem is the three-dimensional analogue of the planar \emph{homotopy area} problem for immersed curves.
But the generalization from $\R^2$ to $\R^3$ presents multiple challenges.
For planar curves with self-intersections, the intersections occur only at isolated points and the curve divides the plane into regions in a simple manner (see Figure \ref{fig:nullhom}).
In contrast, an immersed sphere in $\R^3$ with self-intersections may intersect itself along curves. 
Because of this, the complement of the surface has a more complicated structure and it becomes more challenging to track how the surface separates the regions.

The behavior of homotopies is also different in three dimensions. 
In the plane, one can consider homotopies that move the curve while preserving its orientation along the deformation. 
Such homotopies are called sense-preserving \cite{ChWa2019}.
This definition uses the fact that an oriented planar curve has a well-defined left and right side. 
Thus, one can require that the curve moves consistently to only one side during the deformation.
Extending this idea to the case of immersed spheres in $\R^3$ is less direct. 
A surface does not determine a single direction of motion in the same way. 
During a null homotopy, the intermediate surfaces may develop self-intersections, and different sheets of the surface may pass through each other. 
Because of this, describing a consistent direction of motion requires a more careful definition.

\begin{figure}[ht!]
    \centering
    \includegraphics[width=0.5\linewidth]{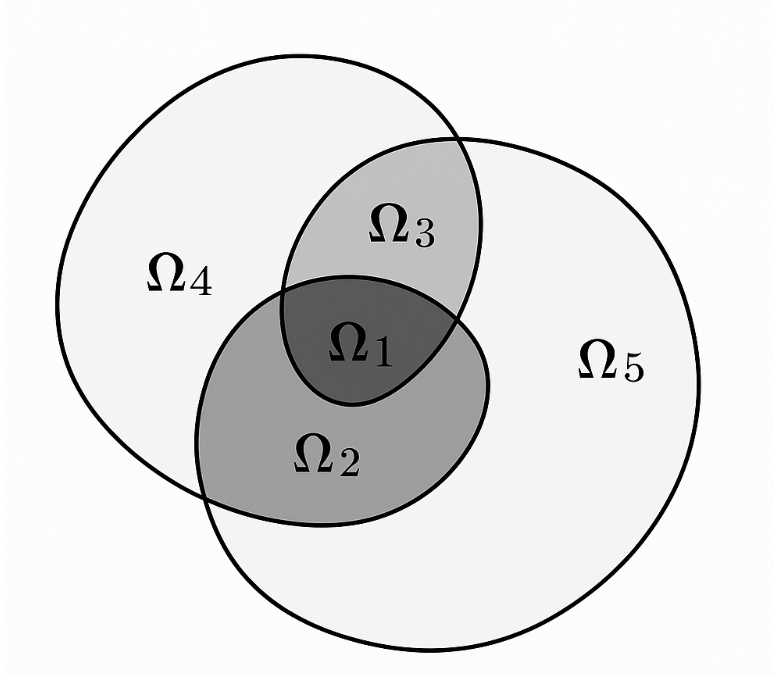}
    \caption{A nontrivial null homotopy of a closed immersed curve in the plane together with the regions $\Omega_1, \dots, \Omega_5$ determined by its image.
      The regions are swept with different multiplicities by a null homotopy of the curve:
      the central region $\Omega_1$ is swept three times, the adjacent regions $\Omega_2$ and $\Omega_3$ are each swept twice, and the outer regions $\Omega_4$ and $\Omega_5$ are swept once each.
      This example illustrates that even in the planar case, minimal swept area may involve nontrivial multiplicities.}
    \label{fig:nullhom}
\end{figure}

For planar curves, intersections appear or disappear through local moves \cite{FaKaWe2017}. 
These events are isolated and can be tracked locally. 
For surfaces in $\R^3$, intersections occur along curves, and during a homotopy these curves may appear, disappear, merge, or split.
Because of this, the structure of the complement may change during the deformation. 
Regions that were separated at one time may become adjacent later as sheets move past each other. 
As a result, tracking how the surface separates space and how many times each region is swept becomes subtler than in the planar case.

In his 1967 thesis, Blank introduced a geometric encoding for planar immersed curves based on cable systems and associated symbolic words \cite{Bl1967}. 
For a normal closed immersed curve $\gamma:S^1\to\R^2$, a \emph{cable system} for $\gamma$ consists of a collection of pairwise disjoint paths (cables) connecting each bounded face of the complement $\R^2\setminus\gamma(S^1)$ to a fixed exterior base point $p_0$.
 Traversing the curve once and recording in order the indices of the cables it crosses, yields a cyclic word.
Each crossing is assigned a sign determined by orientation. 
A crossing is positive if the curve meets the cable from right to left relative to the cable’s orientation, and negative otherwise.
The resulting cyclic word is called the \emph{Blank word} associated with the cable system. 
This word encodes both intersection data and orientation information of the curve.

A Blank word is said to be \emph{reduced} if it contains no two consecutive symbols corresponding to the same cable with opposite orientations. 
To avoid unnecessary backtracking intersections, one often imposes the \emph{shortest path assumption} where each cable is chosen so that it intersects the curve minimally among all smooth paths connecting its region to the base point. 
Under this condition, the associated Blank word is automatically reduced.
Although many different cable systems can be drawn for the same curve, the reduced Blank word is invariant under cable isotopy, that is, under smooth deformations of the cables that preserve their intersection pattern with the curve.

\subsection{Our Contributions}
We extend the cable system approach from planar curves to immersed spheres in $\R^3$. 
Such an extension is not obvious, since one can walk along the curve in the plane and assign signs according to how crossings are met. 
But this is not possible in three dimensions since even when self-intersections are transverse, they can form curves rather than isolated points, and there is no natural way to walk on the sphere and propagate orientations in the same manner. 
Instead, we show that the Blank cable system of an immersed curve in the plane can in effect be used to encode the degree associated with the connected components determined by an immersed sphere in $\R^3$. 

We introduce the cable index (Definition \ref{def:cableindx}) and show that it captures how the surface separates space. 
This quantity plays a central role in our work. 
In Proposition \ref{prop:cableindex}, we prove that the cable index of each region agrees with its Brouwer degree. 
In particular, the cable index is constant on each component and provides a direct geometric description of the degree.
Using this concept, we define the index function (Definition \ref{def:indxfun}) and the total degree (Definition \ref{def:totdeg}), and prove a lower bound on the swept volume of any Lipschitz null homotopy in Theorem \ref{thm:lowbound}.
This extends the planar theory of winding number and area to three dimensions. 
We also introduce a notion of sense-preserving homotopies for immersed spheres in $\R^3$, extending the classical planar concept to the three-dimensional setting (Definition \ref{prop:snspresrv}). 
In Proposition \ref{prop:minhom} we prove that such homotopies induce a monotone index function, and therefore attain the degree-based lower bound on the swept volume. 
In addition, in Proposition \ref{prop:minhomimrs}, we describe a second, more geometric approach in the case where the immersed sphere arises as the boundary of an immersed ball. 
In this setting, the lower bound is attained by a homotopy obtained by deforming the unit ball and then applying the extension map. 
This gives a direct interpretation of the swept volume through the area formula.

Our framework also applies to the planar case of curves, where the cable index reduces to the classical winding number. 

We provide a linear-time algorithm for computing the cable index, and hence the degree associated with each connected component determined by the surface (see Subsection \ref{subsec:algorithm}). 
Unlike the classical definition of degree, which is based on counting preimages or integrating over the surface, our method computes the degree directly from the cable system by walking along the cables and tracking how the index changes across regions. 
Existing methods assume a given map and compute its degree by subdividing the domain and evaluating the function on the boundary. 
In contrast, our method does not take a map as input. 
It computes the degree directly from the combinatorial structure of the cable system, associated with an immersed surface, and does not rely on geometric approximation or numerical integration and avoids subdivision and function evaluation. 
Unlike existing algorithms, our method runs in linear time and applies both to immersed curves in the plane and to immersed spheres in $\R^3$, providing a unified way to compute degree information in both settings. 
Moreover, the algorithm computes the degree associated with all regions, using a single representative point per region. 
This computation is exact and depends only on the combinatorial structure of the cable system.
 
\subsection{Related Work}

In the plane, Chambers and Wang defined the homotopy area between two simple homotopic curves and proved a sharp lower bound in terms of the winding numbers of the regions \cite{ChWa2013}.  
In particular, the planar lower bound becomes an equality when the homotopy is sense-preserving.
Nie \cite{Ne2014} further studied the minimum homotopy area of planar curves using algebraic constructions of the associated word. 
In particular, Nie showed how to compute the minimum homotopy area by applying dynamic programming to the combinatorial word representation, and related this approach to structures arising in geometric group theory.
Chang, Fasy, McCoy, and Millman \cite{ChFaMcMiWe2023} gave conditions under which a reduced Blank word is uniquely determined up to cyclic relabeling.
Building on this framework, they introduced the weighted cancellation norm, defined as the minimal total area of regions that remain after all possible combinatorial cancellations. 
During a null homotopy of a normal curve, opposite regions may cancel. 
They showed that this norm equals the minimal swept area of any null homotopy of a normal planar immersed curve.
They also provided an algorithm to compute this quantity. 
Starting from a valid cable system and the associated Blank word, their method uses dynamic programming to determine the weighted cancellation norm efficiently, and thus the minimal swept area.
They also showed that Blank’s and Nie’s word constructions agree under suitable assumptions and gave a geometric interpretation of Nie’s dynamic programming algorithm.

For surfaces, winding number type quantities have been studied in geometry processing. 
Jacobson, Kavan, and Sorkine-Hornung \cite{JaKaHo2013} introduced generalized winding numbers for triangle meshes in $\R^3$, computed using solid angle formulas.
Their method evaluates a winding number function over space.
The running time depends on both the number of mesh elements and the number of points where the function is evaluated. 
A direct computation takes linear time in the number of facets for each evaluation.
They also propose a hierarchical method that speeds up the computation in practice, although no explicit asymptotic bound is stated.
Barill et al.~\cite{BaDiScLeJa2018} extend this approach to triangle soups and point clouds in $\R^3$ where the winding number is evaluated from discrete geometric representations of the surface.
As in the previous setting, direct computation takes linear time in the number of triangles.
Their method improves this in practice, with evaluation times that show close to logarithmic behavior.
These approaches rely on geometric or numerical evaluation of the surface.

From the viewpoint of geometric measure theory, the existence of volume-minimizing fillings is guaranteed in the setting of integral currents \cite{Fe1969, Al2000, FeFl1960}. 
Related questions also appear in filling problems, such as Gromov's filling volume \cite{Gr1983}. 
These problems are static in nature, focusing on minimal fillings of a boundary, rather than the volume swept during a deformation.
White \cite{Wh1983} proved existence results for least-area mappings under embedding assumptions. 
In contrast, we study immersed spheres in $\R^3$, which may have self-intersections. 

Since we view the Brouwer degree as a generalization of the winding number in the plane, and compute it as part of finding the lower bound on the swept volume, we briefly mention existing methods for computing it.
We also mention algorithms for the swept area, since our approach applies to the planar setting as well.

Nie \cite{Ne2014} described a polynomial-time algorithm for computing the minimum homotopy area of planar curves based on dynamic programming over a combinatorial word representation. 
Fasy, Karako{\c c}, and Wenk \cite{FaKaWe2017} study minimum homotopy
area using a recursive decomposition of the curve. 
Their approach splits the curve at intersection points and breaks it into subcurves. 
In general, this leads to an exponential-time algorithm, since
all possible decompositions may need to be considered. 
Chambers and Wang compute the minimum homotopy area by working directly with the geometry of the curves \cite{ChWa2019}. 
Their algorithm splits the curves at intersection points and considers subcurves between intersections. 
For pairs of intersections, it defines subproblems and computes winding numbers on these pieces. 
The method uses dynamic programming, with running time that can be quadratic or close to linear, depending on the setting of the case.
Chang, Fasy, McCoy, Millman, and Wenk \cite{ChFaMcMiWe2023} compute the minimum homotopy area of planar curves using a combinatorial encoding. 
Their method represents the curve as a word derived from the cable system, where each region corresponds to a letter and opposite regions cancel. 
The problem is then reduced to that of finding an optimal sequence of cancellations. 
They use a dynamic programming algorithm over substrings of the word, which leads to a polynomial-time algorithm.

Early approaches for degree computation include the method for computing the Brouwer degree in the plane by approximating the image of a curve and counting crossings of a fixed direction and the work of O’Neil and Thomas \cite{OnTh1975}, which is based on approximating a multidimensional integral related to the degree. 
Stenger \cite{St1975} proposed a method based on a simplicial decomposition of the boundary, computing the degree from the signs of the function evaluated at the vertices.
Boult and Sikorski \cite{BoSi1986} developed a different method to calculate the degree and studied the complexity of computing the degree for Lipschitz functions, providing algorithm together with upper and lower bounds on the number of function evaluations. 
Aberth \cite{Ab1994} introduced a method based on interval arithmetic.
Nakamura and Murashige \cite{NaMu2006} proposed a different approach based on computational homology theory. 
More recently, Franek and Ratschan \cite{FrRa2015} developed an interval-based algorithm for computing the degree of a continuous function defined on a region that is a product of intervals.
Their method applies to functions given by arithmetic expressions and does not require differentiability or a known Lipschitz constant. 
It separates the computation into a numerical step and a combinatorial step that computes the degree. 
They also note that performance typically worsens as the dimension increases, while the method works well in low dimensions.

\section{Background and Definitions}
Our motivation for the present work comes from the planar setting.
In that setting, the minimal homotopy area problem can be studied by analyzing how the curve separates the regions of the plane.  
This suggests that the topological information controlling homotopy area can be recovered from a discrete encoding of the structure of the complement.  
In the plane, the relevant topological quantity is the winding number of the curve around a point.

\begin{definition} \label{def:wndnum}
Let $\gamma : S^1 \to \C^2$ be an $C^1$ oriented closed curve and let $x \notin \gamma(S^1)$.
The winding number of $\gamma$ around $x$ measures how many times the curve winds around the point $x$ \cite{De1985}.
It can be defined by
\[
\mathrm{w}(\gamma,x)
=
\frac{1}{2\pi i}
\int_{\gamma}
\frac{dz}{z-x}.
\]
Equivalently, the winding number counts the number of times the direction from $x$ to $\gamma(t)$ rotates as $t$ travels once around the curve, with sign determined by the orientation of the curve.
\end{definition}

A natural candidate for a three-dimensional analogue is the degree of an associated map, a notion that was introduced by  Brouwer \cite{Br1911}.

\begin{definition} \label{def:Brdeg}
Let $\Omega \subset \R^n$ be a bounded open set and let 
$f : \overline{\Omega} \to \R^n$ be continuous. 
If $y \notin f(\partial \Omega)$, the \emph{Brouwer degree} of $f$ with respect to $\Omega$ and $y$, denoted 
$\deg(f,\Omega,y)$, is an integer that measures the algebraic number of preimages of $y$ under $f$ \cite{De1985}. 
Intuitively, the degree records how many times the map $f$ covers the point $y$, counting multiplicity and orientation.

More precisely, when $f$ is differentiable and $y$ is a regular value, the degree can be computed as
\[
\deg(f,\Omega,y) = \sum_{x \in f^{-1}(y)} \operatorname{sign}(\det Df(x)),
\]
where $Df(x)$ is the Jacobian matrix of partial derivatives of $f$ at $x$, and $\det Df(x)$ is its determinant. 
The sign of this determinant records whether the map locally preserves or reverses orientation at the point $x$.
\end{definition}

The Brouwer degree is invariant under homotopies that avoid the boundary, and it satisfies an additivity property with respect to decompositions of the domain. 
In particular, if $\Omega_1$ and $\Omega_2$ are disjoint open subsets of $\Omega$ and 
$y \notin f(\overline{\Omega} \setminus (\Omega_1 \cup \Omega_2))$, then
\[
\deg(f,\Omega,y) = \deg(f,\Omega_1,y) + \deg(f,\Omega_2,y).
\]
These properties make the degree a natural tool for studying how a map wraps a domain around a point.

Throughout this paper, we consider a smooth immersion
$\iota : S^2 \to \R^3$, that is, the differential $d\iota_p$ is injective for every $p\in S^2$ \cite{Le2013}. 
We also assume that $\iota$ extends to a Lipschitz map  $f : B^3 \to \R^3$ that preserves orientation of the ball $B^3$ and agrees with $\iota$ on its boundary.

Locally, $\iota(S^2)$ is a smooth embedded surface, although global self-intersections may occur. 
This means that whenever two sheets of the surface meet at a point $x \in \R^3$, their tangent planes together span $\R^3$ \cite{Le2013}. 
We assume that all self-intersections are transverse. Away from these curves, the image is a smooth embedded surface. 
In $\R^3$, transverse self-intersections occur along smooth curves, and triple intersections, when present, are isolated points. 
In addition, we assume that there are only finitely many self-intersection curves and triple intersection points.

\begin{definition} \label{def:hom}
A \emph{homotopy} is a family of maps $h_t : X \to Y$, indexed by $t \in [0,1]$, such that the associated map
$H : X \times [0,1] \to Y, \ H(x,t)=h_t(x)$ 
is continuous. 
The maps $h_0$ and $h_1$ are called the initial and terminal maps. 
Two maps $f,g : X \to Y$ are said to be homotopic if there exists a homotopy $h_t$ with $h_0=f$ and $h_1=g$.
A \emph{null homotopy} is a continuous map where $h_1(x)$ is the constant map \cite{Hu1959}.

A null homotopy of an immersed sphere $\iota : S^2 \to \R^3$ is a map
\[
H : S^2 \times [0,1] \to \R^3, 
\qquad H(p,t) = h_t(p),
\]
such that $h_0=\iota$ and $h_1$ is constant. Geometrically, it represents a continuous contraction of $\iota(S^2)$ to a point.
\end{definition}

Throughout this work we restrict to Lipschitz null homotopies. We also assume that each point $x \in \R^3$ except the terminal point, is covered only finitely many times during the homotopy, meaning that for every $x \in \R^3$, the preimage
\[
H^{-1}(x)=\{(p,t)\in S^2\times[0,1]: H(p,t)=x\}
\]
is finite. 

\begin{definition}
A homotopy $H$ is \emph{Lipschitz} \cite{EvGa2015}
if there exists a constant $L>0$ such that
\[
|H(x,t)-H(y,s)| \le L\, |(x,t)-(y,s)|
\]
for all $(x,t),(y,s)\in S^2\times[0,1]$ .
\end{definition}

If $H$ is Lipschitz, we define the \emph{swept volume} of $H$ using the area formula \cite{Fe1969}
\[
\operatorname{Vol}(H)
= \int_{S^2 \times [0,1]} JH(x,t)\, dx\, dt,
\]
where $JH$ denotes the $3$-dimensional Jacobian of $H$.
In particular, the volume accounts for how many times each point is covered.

We collect notation used in this paper in Table 1.

\begin{table}[ht!] 
  \centering
  \caption{Notation used throughout the paper.}
\begin{tabular}{ll}
\hline
\textbf{Notation} & \textbf{Description} \\
\hline

$\mathcal{H}^k$ & $k$-dimensional Hausdorff measure \\

$S^2$ & Unit sphere in $\R^3$ \\

$B^3$ & Unit ball in $\R^3$ \\

$\iota : S^2 \to \R^3$ & Smooth immersed sphere \\

$\Sigma = \iota(S^2)$ & Image of the immersed sphere \\

$f : B^3 \to \R^3$ & Lipschitz extension of $\iota$ to the ball \\

$\Omega_\infty$ & Unbounded component of $\R^3 \setminus \Sigma$ \\

$\Omega_i$ & Bounded components of $\R^3 \setminus \Sigma$ \\

$p_i \in \Omega_i$ & Interior reference point in $\Omega_i$ \\

$p_\infty \in \Omega_\infty$ & Exterior base point \\

$\pi_i : [0,1] \to \R^3$ & Oriented cable from $p_i$ to $p_\infty$ \\

$\dot{\pi}_i(q)$ & Tangent vector of $\pi_i$ at crossing $q$ \\

$\nu(q)$ & Oriented unit normal to $\Sigma$ at $q$ \\

$\operatorname{sgn}(q)$ & Crossing sign $\mathrm{sign}\langle \dot{\pi}_i(q), \nu(q)\rangle$ \\

$C(\Omega_i)$ & Cable index of region $\Omega_i$ \\

$\operatorname{ind}_\Sigma(x)$ & Index function on $\mathbb{R}^3 \setminus \Sigma$ \\

$D(\Sigma)$ & Total degree: $\sum_i C(\Omega_i)\,\mathrm{Vol}(\Omega_i)$ \\

$V_{\deg}(\iota)$ & Degree-weighted volume \\

$H : S^2 \times [0,1] \to \R^3$ & Lipschitz null-homotopy \\

$\Sigma_t$ & Intermediate surface $H_t(S^2)$ \\

$\mathrm{Vol}(H)$ & Swept $3$-volume of $H$ \\

$N(x,H)$ & Number of preimages of $x$ under $H$ \\

$\xi_t(p)$ & Orientation of $h_t$ at $H(p,t)=x$ \\

$W = [w_1,\dots,w_N]$ & Cable-word \\

$(\Omega_a,\Omega_b)$ & Symbol for oriented crossing from $\Omega_a$ to $\Omega_b$ \\

$n_i$ & Reduced coefficient equal to $C(\Omega_i)$ \\

\hline
\end{tabular}
\end{table}

\section{Main Results}
In this section we introduce the cable-system construction for immersed spheres and establish its connection with the Brouwer degree. 
We then derive a lower bound for the swept volume of Lipschitz null homotopies and show that under a monotonicity condition on the index functions, this lower bound is attained.
In the end, we describe a linear-time algorithm for computing the associated indices.

\subsection{Construction of Cable Systems for Immersed Spheres}
We begin with the topology of the complement.

Let $\iota : S^2 \to \R^3$ be a smooth immersion with finitely many transverse self-intersections, and set $\Sigma=\iota(S^2)$. 
Since $S^2$ is compact, its image $\Sigma$ is compact in $\R^3$, and therefore the complement $\R^3 \setminus \Sigma$ is open and contains an unbounded component.
Locally, near each point of $\Sigma$, the surface consists of finitely many smooth sheets meeting transversely. 
In a sufficiently small neighborhood, these sheets divide space into finitely many regions. 
Because there are only finitely many self-intersection curves and triple points, this local finiteness extends globally. 
Hence, the complement decomposes into finitely many open, path-connected regions. 
Exactly one of these is unbounded; we denote it by $\Omega_\infty$, and we write the bounded regions as $\Omega_1,\dots,\Omega_N$.

We now adapt Blank’s cable system construction to immersed spheres in $\R^3$.
In the planar case, one traverses the curve and records the cables it crosses, producing a cyclic word. 
This works because a curve is one-dimensional and walking once around it captures all intersections in order.
For an immersed sphere, this approach is no longer available. 
The surface is two-dimensional, and there is no canonical way to “walk” on it and record crossings globally. 
Self-intersections occur along curves rather than isolated points, and the planar combinatorial encoding does not directly generalize.
Instead, we reverse the viewpoint. 
Rather than walking along the surface, we consider the cables themselves. 
Each bounded region is connected to the exterior by a cable, and we record how this cable crosses the immersed sphere. 
The orientation of the sphere determines a unit normal on its regular part, and this allows us to assign a sign to each crossing of a cable with the surface.
Thus, to each region we associate an integer obtained by summing the signed crossings of its cable. 
Unlike the planar case, this construction does not produce a cyclic word. 
Instead, it produces a collection of integers, one for each region, which will serve as the basis for our results.

Let $\Sigma=\iota(S^2)$ and let 
\[
\R^3 \setminus \Sigma
=
\Omega_\infty \;\cup\; \Omega_1 \;\cup \cdots \cup\; \Omega_N
\]
be the decomposition of the complement.
For each bounded region $\Omega_i$, choose an interior point $p_i \in \Omega_i$, and fix a point $p_\infty \in \Omega_\infty$.
A \emph{cable} associated to $\Omega_i$ is a smooth simple path
\[
\pi_i : [0,1] \to \R^3
\]
such that
\[
\pi_i(0)=p_i, 
\qquad
\pi_i(1)=p_\infty,
\]
and $\pi_i$ intersects $\Sigma$ transversely, only at regular points, and finitely many times.
We orient each cable $\pi_i$ from $p_i$ toward $p_\infty$. 

The domain $S^2$ carries its standard orientation. 
On the regular part of the immersed surface $\Sigma$, this induces a smooth unit normal vector field $\nu$.
At a transverse self-intersection point, the image consists locally of two smooth sheets meeting along a curve. 
Each sheet inherits its orientation from the domain sphere and therefore has its own consistently defined normal vector. 

Let $q \in \pi_i \cap \Sigma$ be a transverse intersection point.
Let $\dot{\pi}_i(q)$ denote the tangent vector to the cable at $q$, and let $\nu(q)$ be the oriented unit normal to $\Sigma$ at $q$.
We define
\[
\operatorname{sgn}(q)
=
\operatorname{sign}\langle \dot{\pi}_i(q), \nu(q) \rangle
\in \{+1,-1\}.
\]

\begin{figure}[ht!]
    \centering
    \includegraphics[width=0.4\textwidth]{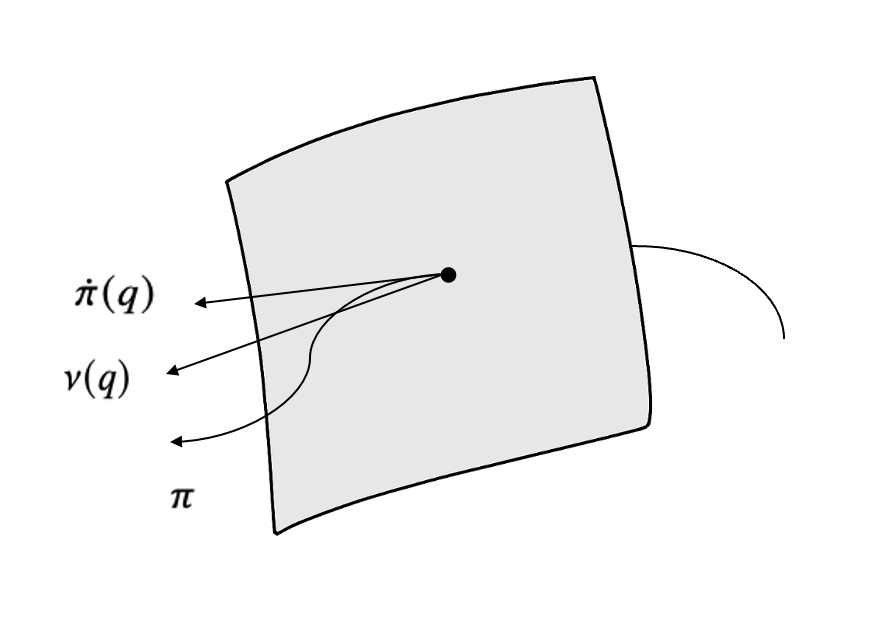}
    \caption{Local geometry at an intersection point $q \in \pi \cap \Sigma$.}
    \label{fig:local_intersection}
\end{figure}
Since we require that each cable intersects $\Sigma$ only at regular points of the immersion, this ensures that the oriented unit normal vector is uniquely defined at every crossing.

For a region $\Omega_i$, and its interior point, define the \emph{cable index} with respect to a point $p_i$ as
\[
C(p_i,\Omega_i)
=
\sum_{q \in \pi_i \cap \Sigma}
\operatorname{sgn}(q).
\]
This integer records the algebraic number of times the cable crosses the immersed sphere.

The family of oriented paths $\{\pi_i\}_{i=1}^N$ defines a \emph{cable system} for the immersed sphere.

\begin{example}
  Cable $\pi_7$ from \autoref{fig:cable_system} begins in $\Omega_7$, crosses the wall separating $\Omega_7$ from $\Omega_3$,
  goes out to $\Omega_{\infty}$,
  crosses into $\Omega_3$ again, and eventually exits.  
  The signed index contributions are 
  $
  +1 \text{(entering $\Omega_3$)},
  +1 \text{(entering $\Omega_{\infty}$)},
  -1 \text{(re-entering $\Omega_3$)},\,\text{and}\,
  +1 \text{(final exit)},
  $
  giving
  $
  \mathrm{C}({p}_7,\Omega_7)=2.
  $
\end{example}

\begin{figure}[ht!]
    \centering
    \includegraphics[width=0.9\textwidth]{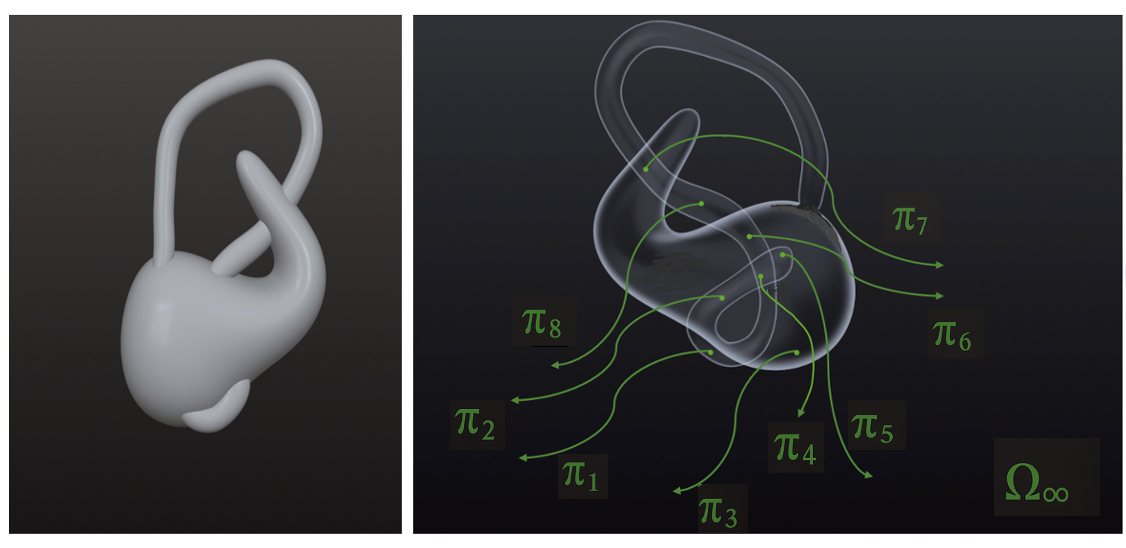}
    \caption{An immersed $2$-sphere in $R^3$ with eight bounded regions, cables $\pi_1,\dots,\pi_8$, and exterior region $\Omega_{\infty}$.
    The sphere begins as a standard round sphere, but at one place the surface is smoothly stretched outward and then stopped.  
    At a different location, the sphere is stretched again, this time forming a long smooth tube that passes through the first stretched region, creating a transverse self-intersection.  
    The tube continues inward, curves around inside the sphere, cuts through the surface a second time, and then terminates smoothly. }
    \label{fig:cable_system}
\end{figure}

\subsection{The Cable Index and the Brouwer Degree}
We now relate the cable index to the Brouwer degree. 

\begin{definition} \label{def:cableindx}
Let $\Sigma \subset \R^3$ be an oriented immersed sphere, and let $\Omega$ be a connected component of $\R^3 \setminus \Sigma$, and let $y \in \Omega$. 
Choose a cable $\pi : [0,1] \to \R^3$ from $y$ to the exterior base point $p_\infty$. 
The \emph{cable index} of $\Omega$ is defined by
\[
C(\Omega) \equiv \sum_{q \in \pi \cap \Sigma} 
\operatorname{sign}\langle \dot{\pi}(q), \nu(q) \rangle,
\]
where $\nu(q)$ is the oriented unit normal to $\Sigma$ at the regular intersection point $q$.
\end{definition}

\begin{proposition} \label{prop:cableindex}

Let $B^3$ be a closed $3$-ball with $\partial B^3 = S^2$, equipped with its standard orientation, and let
\[
f : B^3 \to \R^3
\]
be a $C^1$ map whose restriction to $\partial B^3$ agrees with the immersion $\iota$ and which preserves orientation of the ball.

Then the cable index of $\Omega$ equals the Brouwer degree of $f$ at $y$, that is,
\[
C(\Omega) = d(f, B^3, y).
\]
In particular, $C(\Omega)$ is constant on $\Omega$.
\end{proposition}

\begin{proof}
Fix $y\in \Omega$ and choose a cable $\pi:[0,1]\to\R^3$ with $\pi(0)=y$ and $\pi(1)=p_\infty\in\Omega_\infty$. 
After a small perturbation, we may assume that $\pi$ intersects $\Sigma=f(\partial B^3)$ transversely, only at regular points of $\iota$, and at distinct parameters
\[
0<t_1<\cdots<t_k<1,
\qquad q_j=\pi(t_j)\in\Sigma .
\]

Set $y_t=\pi(t)$. Since $y_t\notin f(\partial B^3)$ for $t\neq t_j$, the degree $d(f,B^3,y_t)$ is defined and is constant on each interval
$(0,t_1)$, $(t_1,t_2)$, \dots, $(t_k,1)$.

At a crossing time $t_j$, the path $y_t$ passes through the oriented surface $\Sigma$ transversely at $q_j$. The degree changes by $\pm 1$ when crossing $f(\partial B^3)$, with sign determined by the direction of crossing relative to the oriented normal. 
For $\varepsilon>0$ small, at a crossing $q_j$, let $y_{j-\varepsilon}$ and $y_{j+\varepsilon}$ be points
immediately before and after the crossing along $\pi$.
So, they lie in the two regions adjacent to $\Sigma$ at $q_j$.
Then
\[
d(f,B^3,y_{j-\varepsilon})-d(f,B^3,y_{j+\varepsilon})
=
\operatorname{sgn}(q_j).
\]
For $t$ near $1$, the point $y_t$ is in the exterior region, so $d(f,B^3,y_t)=0$. 
Summing the jumps along $\pi$ from $t=0$ to $t=1$ gives
\[
d(f,B^3,p_\infty)
=
d(f,B^3,y)
-
\sum_{j=1}^k \operatorname{sgn}(q_j).
\]
Since $d(f,B^3,p_\infty)=0$, we obtain
\[
d(f,B^3,y)
=
\sum_{j=1}^k \operatorname{sgn}(q_j)
=
C(\Omega).
\]
Since the Brouwer degree is constant on each connected component of $\R^3\setminus f(\partial B^3)=\R^3\setminus\Sigma$, the same value holds for every $y\in\Omega$. Hence $C(\Omega)$ is constant on $\Omega$.
\end{proof}

\subsection{The Lower Bound Theorem}  

\subsubsection{Proof via the Cable Index} 
\label{subsec:cablindxproof}
Chambers and Wang established that for two simple open homotopic curves, the swept area of any homotopy is bounded below by the total winding number area \cite{ChWa2019}.
We extend this principle to immersed spheres in $\R^3$.

In the three-dimensional setting, the winding number of planar regions is replaced by the cable index (equivalently, the Brouwer degree) associated with each connected component of the complement.
This degree counts, with sign, how many times the immersed sphere wraps around a point in space.

\begin{definition} \label{def:indxfun}
For $x \in \R^3 \setminus \Sigma$,
let $\Omega$ be the connected component containing $x$.
We define the index function
\[
\operatorname{ind}_\Sigma(x) \equiv C(\Omega),
\]
where $C(\Omega)$ is the cable index of the region. 
\end{definition}

The above definition is well defined.
Since the Brouwer degree is constant on each connected component of
$\R^3 \setminus \Sigma$,
the value does not depend on the choice of point or cable.
Thus, $\operatorname{ind}_\Sigma(x)$ depends only on the region containing $x$.
The function $\operatorname{ind}_\Sigma$ is undefined on $\Sigma$ itself,
since the Brouwer degree is defined only away from the boundary image.

\begin{definition} \label{def:totdeg}
We now define the total degree (or degree-weighted volume) of $\Sigma$ by
\[
\mathcal{D}(\Sigma)
=
\int_{\R^3}
\operatorname{ind}_\Sigma(x)\, d\mathcal{H}^3(x)
=
\sum_{\Omega}
\operatorname{ind}_\Sigma(\Omega)\,
\operatorname{Vol}(\Omega),
\]
where the sum runs over all connected components
$\Omega \subset \R^3 \setminus \Sigma$.
\end{definition}

This quantity measures the total signed volume enclosed
by the immersed sphere, counted with multiplicity.
It is the three-dimensional analogue of the total winding number area for immersed curves in the plane.

\begin{theorem} \label{thm:lowbound}
Let $\iota : S^2 \to \R^3$ be a smooth immersion with finitely many pairwise transverse self-intersections, and let 
\[
H : S^2 \times [0,1] \to \R^3
\]
be a Lipschitz null homotopy with $H_0=\iota$ and $H_1$ constant. 
Then
\[
\operatorname{Vol}(H)\ \ge\ \bigl|\mathcal{D}(\Sigma)\bigr|,
\]
where $\mathcal{D}(\Sigma)$ is the total degree of the immersed sphere.
\end{theorem}
\begin{proof}
Fix $x\in\R^3$. For each $t\in[0,1]$ with $x\notin \Sigma_t$ set
\[
F(t)\equiv \operatorname{ind}_{\Sigma_t}(x).
\]
By assumption, the set
\[
T_x \equiv \{t\in[0,1]: x\in \Sigma_t\}
\]
is finite, so $F$ is defined on $[0,1]\setminus T_x$.

On each connected component of $[0,1]\setminus T_x$ the point $x$ stays in a fixed region of $\R^3\setminus \Sigma_t$. Hence $\operatorname{ind}_{\Sigma_t}(x)$ does not change there, and $F$ is constant on each such interval. 
Therefore, $F$ is piecewise constant with finitely many jump times.
Each time $t_0\in T_x$ corresponds to the surface passing through $x$ where several sheets of the surface may pass through $x$.
Each preimage $p \in S^2$ with $H(p,t_0)=x$ contributes at most one unit to the change of the index. 
In particular $F$ has bounded variation \cite{EvGa2015} and
therefore,
\[
|F(t_0^+) - F(t_0^-)| \le \#\{p \in S^2 : H(p,t_0)=x\}.
\]
where denote by $F(t_0^-)$ and $F(t_0^+)$ the left and right limits of $F$ at $t_0$.

Summing over all such times gives
\[
\operatorname{Var}(F) \le N(x,H).
\]
where
\[
N(x,H) \equiv \#\{(p,t)\in S^2\times[0,1]: H(p,t)=x\}
\]
counts the preimages of $x$ under $H$.

Since $H_1$ is constant, we have $F(1)=0$. 
Also $F(0)=\operatorname{ind}_{\Sigma_0}(x)$. 
Hence
\[
|\operatorname{ind}_{\Sigma}(x)|=|F(0)-F(1)|\le \operatorname{Var}(F)\le N(x,H).
\]
Integrating over $\R^3$ gives
\[
\int_{\R^3} |\operatorname{ind}_{\Sigma}(x)|\,dx
\ \le\
\int_{\R^3} N(x,H)\,dx.
\]
By the area formula, the right-hand side equals $\operatorname{Vol}(H)$. Therefore
\[
\operatorname{Vol}(H)\ \ge\ \int_{\R^3} |\operatorname{ind}_{\Sigma}(x)|\,dx
\ \ge\ \Bigl|\int_{\R^3} \operatorname{ind}_{\Sigma}(x)\,dx\Bigr|
=|\mathcal{D}(\Sigma)|. \qedhere
\]
\end{proof}

\begin{corollary}
For every admissible Lipschitz null homotopy $H$ of $\Sigma$,
$\operatorname{Vol}(H)\ge |\mathcal{D}(\Sigma)|$.
Hence,
\[
\inf_H \operatorname{Vol}(H)\ge |\mathcal{D}(\Sigma)|.
\]
\end{corollary}

\begin{remark}
In Theorem \ref{thm:lowbound} we assume that for every fixed $x \in \R^3$, the preimage $H^{-1}(x)$ is finite. 
In particular, the set of times $T_x$ is finite. 
Thus, the function $F(t) = \operatorname{ind}_{\Sigma_t}(x)$
has only finitely many jump discontinuities and is constant between them.
In particular, $F \in BV([0,1])$.
\end{remark}

\begin{definition} \label{prop:snspresrv}
Let $\iota : S^2 \to \R^3$ be an oriented immersion and let 
$H : S^2 \times [0,1] \to \R^3$ be a Lipschitz homotopy with $\Sigma_t = H_t(S^2)$. 
We say that $H$ is \emph{sense-preserving} if for every regular crossing $H(p,t)=x$, the quantity
\[
\xi_t(p) \equiv \langle \partial_t H(p,t), \nu_t(p) \rangle
\]
where $\nu_t(p)$ is the oriented unit normal to $\Sigma_t$ at the regular intersection point $x$, has the same sign. 
In other words, all sheets of the surface pass through 
points in a consistent direction with respect to the chosen orientation.
\end{definition}

\begin{remark}
At regular points, the sign of $\langle \partial_t H(p,t), \nu_t(p) \rangle$ agrees with the sign of the Jacobian of $H$. Thus, the sense-preserving condition can be viewed as requiring that the map $H$ has a consistent orientation wherever the surface is regular.
\end{remark}

Each surface $\Sigma_t = H_t(S^2)$ inherits its orientation from the orientation of $S^2$ through the map $H_t$. 
At points where the surface intersects itself, different
sheets may pass through the same point in space, but each sheet still has its own well-defined normal direction coming from the sphere.
This lets us compare how the surface moves with its orientation by looking at $\langle \partial_t H(p,t), \nu_t(p) \rangle$, which tells us whether the surface is moving in the positive or negative normal direction.

\begin{proposition} \label{prop:minhom}

Let $\iota$ and $H$ satisfy the assumptions of Theorem~\ref{thm:lowbound}.
Fix $x \in \R^3$ and for all $t$ with $x \notin \Sigma_t$ extend the index function $F$ continuously from the right at each crossing time,
\[
F(t) = \lim_{s \to t^+} \operatorname{ind}_{\Sigma_s}(x),
\qquad \text{whenever } x \in H_t(S^2).
\]
Then $F$ is piecewise constant and right-continuous on $[0,1]$. 
If $H$ is sense-preserving, then
\[
|\operatorname{ind}_{\Sigma}(x)| = \operatorname{Var}(F) = N(x,H).
\]
\end{proposition}

\begin{proof}
Fix $x \in \R^3$ and let
\[
T_x=\{t\in[0,1]:x\in \Sigma_t\}.
\]
We know that this set is finite. 
List its elements as
\[
0<t_1<\cdots<t_m<1.
\]
For $t\notin T_x$, the point $x$ stays in one connected component of
$\R^3\setminus \Sigma_t$. 
Hence the index
\[
F(t)=\operatorname{ind}_{\Sigma_t}(x)
\]
is constant between two consecutive crossing times. With the right-continuous convention at the crossing times, $F$ is piecewise constant and right-continuous on $[0,1]$.

Now assume that $H$ is sense-preserving. 
At each crossing time $t_i$, every sheet passing through $x$ moves through $x$ in the same oriented direction. 
Therefore all contributions to the jump of $F$ have the same sign. 
Thus the size of the jump is exactly the number of preimages of $x$ at that time:
\[
|F(t_i^+)-F(t_i^-)|=
\#\{p\in S^2:H(p,t_i)=x\}.
\]
Summing over all crossing times gives
\[
\operatorname{Var}(F)
=
\sum_{i=1}^m |F(t_i^+)-F(t_i^-)|
=
\sum_{i=1}^m \#\{p\in S^2:H(p,t_i)=x\}
=
N(x,H).
\]
Since $H$ is sense-preserving, the jumps of $F$ all have the same sign, so $F$ is monotone. 
Also, $H_1$ is constant, hence $F(1)=0$, while
\[
F(0)=\operatorname{ind}_{\Sigma}(x).
\]
Therefore
\[
\operatorname{Var}(F)=|F(0)-F(1)|
=
|\operatorname{ind}_{\Sigma}(x)|.
\]
Combining the two equalities gives
\[
|\operatorname{ind}_{\Sigma}(x)|=\operatorname{Var}(F)=N(x,H). \qedhere
\]
\end{proof}

\begin{corollary}

If $H$ is sense-preserving, then for almost every $x \in \R^3$,
\[
|\operatorname{ind}_{\Sigma_0}(x)| = N(x,H).
\]
Consequently,
\[
\operatorname{Vol}(H)
=
\int_{\R^3} N(x,H)\, d\mathcal{H}^3(x)
=
\int_{\R^3} |\operatorname{ind}_{\Sigma}(x)|\, d\mathcal{H}^3(x)
=
|\mathcal{D}(\Sigma)|.
\]
In particular, such a homotopy attains the lower bound of Theorem~\ref{thm:lowbound}.

\end{corollary}

\subsubsection{A Minimizing Homotopy from an Immersed Ball} \label{subsec:imrsball}

\bigskip

In this subsection we show that, in the special case where the immersed sphere arises as the boundary of an immersed ball, the lower bound is attained.

\begin{proposition} \label{prop:minhomimrs}

Let $\iota : S^2 \to \R^3$ be a smooth immersion, and assume that $\iota$ arises as the boundary of a smooth, orientation preserving immersion
\[
G : B^3 \to \R^3,
\qquad
G|_{S^2} = \iota.
\]
Then there exists a (smooth) null homotopy
\[
H : S^2 \times [0,1] \to \R^3
\]
such that
\[
\operatorname{Vol}(H)
=
\int_{\R^3} \deg(G,x) \, dx.
\]
In particular, the lower bound is attained. Note that we write $\deg(G,x)$ for the degree of $G$ with respect to $x\in G(B^3)\subset\R^3$.
\end{proposition}

\begin{proof}

\begin{enumerate}
    \item \label{step-1} Every smooth homotopy of the sphere can be written as a smooth map 
    \[H:\partial B^3\times[0,1]\rightarrow\R^3,\] but since we are only interested in null homotopies, \[H(\cdot,1):\partial B^3\rightarrow \{\text{a point}\},\]we need only consider \[H:\left[\{\partial B^3\times[0,1]\}\bmod\{\partial B^3\times\{1\}\}\right]\rightarrow\R^3.\] But, since 
    \[ B^3 = \left[\{\partial B^3\times[0,1]\}\bmod\{\partial B^3\times\{1\}\}\right]\]
    we need only consider continuous maps $G$, $G:\B^3\rightarrow\R^3$.
    \item Let $G$ be the smooth, orientation preserving map such that the immersed sphere, $\iota(\mathcal{S}^2)$, is just $G|_{\partial B^3}$. By Step~\ref{step-1} above, we can identify $G$ with a candidate null homotopy and we can also conclude that 
    \[\deg(G,x) = N(G,x) = \mathcal{H}^0(G^{-1}(x)).\]
    \item Because we know that any other continuous map $\hat{G}$ leaving the boundary fixed has the same degree everywhere off the image of the boundary (i.e., off of $ G|_{\partial B^3} = \iota(\mathcal{S}^2)$) and we know that $N(\hat{G},x) \geq \deg(\hat{G},x)=\deg(G,x)=\mathcal{H}^0(G^{-1}(x))$, 
    \item By the area formula, we conclude 
    \begin{eqnarray*}
    \text{ Volume swept out by $G$} &=& \text{Volume of $G(B^3)$ with multiplicity} \\
         &=& \int_{B^3} JG \;dx \\
         &=&\int_{G(B^3)}  \mathcal{H}^0(G^{-1}(x)) \; dx \\
         &=&  \int_{G(B^3)} \deg(G,x) \; dx  
    \end{eqnarray*}
    and the minimal homotopy sweeps out exactly the integral of the degree as calculated by the cable system. \qedhere
\end{enumerate}
\end{proof}

\subsection{Algorithm for Cable Index Computation} \label{subsec:algorithm}

We now present a linear time algorithm for computing the cable indices of an immersed sphere. 
While the cable index was defined geometrically and identified with the Brouwer degree, the purpose of this section is to show that it can be computed purely combinatorially from a finite cable system.
The algorithm transforms the geometric cable system associated with $\iota(S^2)\subset \R^3$ into a combinatorial cable-word, reduces it using local cancellations and transitive relations, and outputs the signed indices $(n_i,\Omega_i)$ for all bounded regions of $\R^3\setminus \iota(S^2)$. 
This construction encodes the same topological information as the degree formulation and provides a concrete method for evaluating the degree weighted lower bound from Theorem \ref{thm:lowbound}.

\paragraph{Cable-word construction.}

Let $\iota : S^2 \to \R^3$ be a smooth immersion with finitely many pairwise transverse self-intersections, and write
\[
\R^3 \setminus \iota(S^2) = \Omega_\infty \cup \Omega_1 \cup \cdots \cup \Omega_N,
\]
with $\Omega_\infty$ the unbounded component. Choose interior points
\[
p_\infty \in \Omega_\infty, \qquad p_i \in \Omega_i \ (1 \le i \le N),
\]
and for each $i$ fix an oriented cable $\pi_i : [0,1] \to \R^3$ from $p_i$ to $p_\infty$ intersecting $\iota(S^2)$ transversely at regular points.

At each intersection point $q \in \pi_i \cap \iota(S^2)$ let $\nu(q)$ be the oriented unit normal to the immersed sheet and $\dot\pi_i(q)$ the cable tangent. 
Define
\[
\operatorname{sgn}(q)=\operatorname{sign}\langle \dot\pi_i(q),\nu(q)\rangle \in \{+1,-1\}.
\]
The cable index of $\Omega_i$ is
\[
C(\Omega_i)=\sum_{q\in \pi_i\cap \iota(S^2)} \operatorname{sgn}(q).
\]

To compute these indices combinatorially, we encode each crossing symbolically. 
Whenever an oriented cable $\pi_i$ crosses $\iota(S^2)$ from a region $\Omega_k$ into a neighboring region $\Omega_j$, we record the symbol
\[
(\Omega_k,\Omega_j)
\]
if $\operatorname{sgn}=+1$, and its barred version
\[
\overline{(\Omega_k,\Omega_j)}
\]
if $\operatorname{sgn}=-1$.

For each cable $\pi_i$ we form the word
\[
w_i=(\Omega_{i_1},\Omega_{j_1})(\Omega_{i_2},\Omega_{j_2})\cdots(\Omega_{i_{m_i}},\Omega_{j_{m_i}})
\]
listing crossings in order from $p_i$ to $p_\infty$. The collection
\[
W=[w_1,\dots,w_N]
\]
is the cable-word of the immersed sphere.

\paragraph{Reduction rules.}
The cable-word is reduced by two elementary operations.

First, adjacent inverse crossings through the same interface cancel. 
If a cable crosses from $\Omega_i$ to $\Omega_j$ and immediately returns, the two symbols
\[
(\Omega_i,\Omega_j)\,\overline{(\Omega_j,\Omega_i)}
\]
are replaced by the empty word. 
This corresponds to an immediate local backtracking and does not change the algebraic count.

Second, consecutive transitions through adjacent regions combine transitively. 
A sequence
\[
(\Omega_a,\Omega_b)(\Omega_b,\Omega_c)
\]
is replaced by a single composite transition
\[
2(\Omega_a,\Omega_c),
\]
where the coefficient records the sum of the two oriented crossings. 
Opposite orientations cancel in the same way.

Repeated application of these rules within each word $w_i$ collapses successive crossings into a single effective transition. 
Because each reduction step preserves the total signed sum of crossings, the reduced word
\[
w_i' = n_i(\Omega_i,\Omega_\infty)
\]
satisfies
\[
n_i = \sum_{q \in \pi_i \cap \iota(S^2)} \operatorname{sgn}(q) = C(\Omega_i).
\]
Thus, the algorithm produces the correct cable index for every bounded region.

\paragraph{Example.}

Consider the subword
\[
(\Omega_1,\Omega_2)\,\overline{(\Omega_2,\Omega_3)}\,(\Omega_3,\Omega_4)\,
\overline{(\Omega_4,\Omega_5)}\,\overline{(\Omega_5,\Omega_6)}.
\]

The first two symbols represent opposite crossings through $\Omega_2$ and cancel:
\[
(\Omega_1,\Omega_2)\,\overline{(\Omega_2,\Omega_3)} \;\mapsto\; 0(\Omega_1,\Omega_3).
\]

Combining with the next crossing gives
\[
0(\Omega_1,\Omega_3)(\Omega_3,\Omega_4) \;\mapsto\; 1(\Omega_1,\Omega_4).
\]

The last two symbols are both negatively oriented and combine as
\[
\overline{(\Omega_4,\Omega_5)}\,\overline{(\Omega_5,\Omega_6)}
\;\mapsto\; -2(\Omega_4,\Omega_6).
\]

Putting everything together,
\[
(\Omega_1,\Omega_2)\,\overline{(\Omega_2,\Omega_3)}\,(\Omega_3,\Omega_4)\,
\overline{(\Omega_4,\Omega_5)}\,\overline{(\Omega_5,\Omega_6)}
\;\mapsto\;
1(\Omega_1,\Omega_4) - 2(\Omega_4,\Omega_6)
\;\mapsto\;
-1(\Omega_1,\Omega_6).
\]

Hence, the final coefficient is $n=-1$.

\paragraph{Correctness of the reduction.}

We now justify that the reduction algorithm computes the cable index.

We call a cable system
\[
\Pi=\{\pi_i:[0,1]\to\R^3\mid 1\le i\le N\}
\]
\emph{simple} if it satisfies the following:

(i) $\pi_i(0)=p_i\in\Omega_i$ and $\pi_i(1)=p_\infty\in\Omega_\infty$;

(ii) whenever $\pi_i$ passes from a bounded region $\Omega_a$ to a different bounded region $\Omega_b$, it crosses $\iota(S^2)$ exactly once along their common interface, and it never re-enters a bounded region it has already left;

(iii) distinct cables are disjoint except at their common endpoint $p_\infty$.

\begin{proposition} \label{prop:algrm}
Let $\Pi$ be a simple cable system. 
For each bounded region $\Omega_i$, let $w_i$ be the cable word associated with $\pi_i$, and let $w_i'$
be the reduced word obtained by repeatedly applying the cancellation and transitive reduction rules. 
If $w_i' = n_i(\Omega_i,\Omega_\infty)$,
then $n_i = C(\Omega_i)$.
\end{proposition}
\begin{proof}
Fix $i$. 
Each symbol in $w_i$ records one transverse intersection
$q\in\pi_i\cap \iota(S^2)$ together with its sign $\mathrm{sgn}(q)$.
Therefore, the total signed sum of symbols in $w_i$ equals
\[
\sum_{q\in\pi_i\cap\iota(S^2)} \mathrm{sgn}(q)=C(\Omega_i).
\]

The cancellation rule removes a consecutive inverse pair through the same wall, whose two signs sum to zero. 
The transitive rule replaces two consecutive crossings
through adjacent regions by a single composite symbol whose coefficient is the sum of the two signed contributions. 
Hence each reduction step preserves the total signed sum.

After all reductions, $w_i'$ has the form $n_i(\Omega_i,\Omega_\infty)$, so its coefficient
$n_i$ equals the preserved signed sum. 
Thus
\[
n_i=\sum_{q\in\pi_i\cap\iota(S^2)} \mathrm{sgn}(q)=C(\Omega_i). \qedhere
\]
\end{proof}

\begin{remark}
A cable may enter the exterior region $\Omega_\infty$ before its final arrival at $p_\infty$. 
These crossings do not affect the total signed count.
When the cable crosses from a bounded region into
$\Omega_\infty$ and later leaves $\Omega_\infty$ again, the two crossings through the same wall occur with opposite orientations and contribute $+1$ and $-1$ to the algebraic count.

In the cable word this appears as a consecutive pair
\[
(\Omega_a,\Omega_\infty)\,\overline{(\Omega_\infty,\Omega_b)},
\]
whose coefficients sum to zero. 
By the cancellation rule the pair is removed, and the total signed contribution is unchanged. 
\end{remark}

\paragraph{Reduced Cable-Word}
After applying all cancellation and transitive reductions to each
subword $w_i$, we obtain reduced subwords
\[
w_i' = (n_{i_1},\Omega_{i_1}) \cdots (n_{i_k},\Omega_{i_k}),
\]
where each coefficient $n_{i_j}$ is the net signed number of
oriented transitions from $\Omega_{i_j}$ to $\Omega_\infty$
along the cable $\pi_i$.

Collecting all reduced subwords gives the reduced cable-word
\[
\mathcal{W}=(w_1',\dots,w_N')
=(n_1,\Omega_1)\cdots(n_k,\Omega_k),
\]
with
\[
n_i = C(\Omega_i).
\]
Thus, $\mathcal{W}$ records the cable index of every bounded
region in $\R^3\setminus\iota(S^2)$.
Assign to each region $\Omega_i$ the weight $\mathrm{Vol}(\Omega_i)$.
Then the total weighted sum
\[
V_{\deg}(\iota)
=\sum_{i=1}^{N} |n_i|\,\mathrm{Vol}(\Omega_i).
\]
gives a computable lower bound for the swept volume
of any null homotopy of $\iota(S^2)$.

\begin{figure}[ht!]
    \centering
    \includegraphics[width=0.75\linewidth]{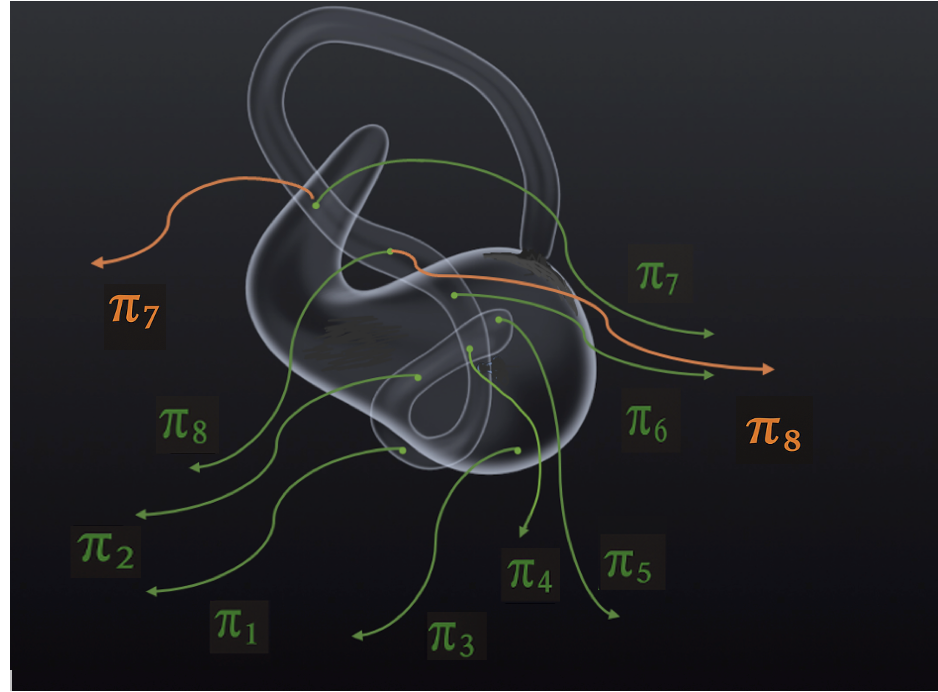}
    \caption{The orange cables represent simple cables.
The green cables illustrate alternative choices: for cable 8 the green path fails to be simple, while for cable 7 both the orange and green versions are simple, although the green one passes through $\Omega_\infty$ earlier than necessary.  
This example shows that different geometric choices of cables may violate the simplicity condition but still produce the same total algebraic index after cancellation.}
    \label{fig:cable_system2}
\end{figure}

\begin{example}
We compute the indices from the green cable system in Figure~5.
Let $\Omega_1,\dots,\Omega_8$ be the bounded regions indicated
by the points $p_1,\dots,p_8$, and let $\pi_1,\dots,\pi_8$
be the green cables.

For each $\pi_i$ we form the cable word $w_i$ by listing,
in order, the oriented crossings with $\iota(S^2)$,
written as symbols $(\Omega_a,\Omega_b)$ whenever the cable
passes from $\Omega_a$ into $\Omega_b$.

\paragraph{Cable $\pi_1$.}
\[
w_1=(\Omega_1,\Omega_\infty).
\]
No reduction is possible, so
\[
w_1'=(\Omega_1,\Omega_\infty),
\qquad
n_1=1=C(\Omega_1).
\]

\paragraph{Cable $\pi_2$.}
\[
w_2=(\Omega_2,\Omega_3)(\Omega_3,\Omega_\infty).
\]
Applying the transitive rule,
\[
(\Omega_2,\Omega_3)(\Omega_3,\Omega_\infty)
\longrightarrow
2(\Omega_2,\Omega_\infty),
\]
so
\[
w_2'=2(\Omega_2,\Omega_\infty),
\qquad
n_2=2=C(\Omega_2).
\]

\paragraph{Cable $\pi_3$.}
\[
w_3=(\Omega_3,\Omega_\infty).
\]
Hence
\[
w_3'=(\Omega_3,\Omega_\infty),
\qquad
n_3=1=C(\Omega_3).
\]

\paragraph{Cable $\pi_4$.}
\[
w_4=(\Omega_4,\Omega_2)(\Omega_2,\Omega_3)(\Omega_3,\Omega_\infty).
\]
First,
\[
(\Omega_4,\Omega_2)(\Omega_2,\Omega_3)
\longrightarrow
2(\Omega_4,\Omega_3),
\]
then
\[
2(\Omega_4,\Omega_3)(\Omega_3,\Omega_\infty)
\longrightarrow
3(\Omega_4,\Omega_\infty).
\]
Thus
\[
w_4'=3(\Omega_4,\Omega_\infty),
\qquad
n_4=3=C(\Omega_4).
\]

\paragraph{Cables $\pi_5,\pi_6$.}
For $i=5,6$,
\[
w_i=(\Omega_i,\Omega_{j(i)})(\Omega_{j(i)},\Omega_\infty).
\]
Applying the transitive rule gives
\[
w_i'=2(\Omega_i,\Omega_\infty),
\qquad
n_i=2=C(\Omega_i).
\]

\paragraph{Cable $\pi_7$.}
\[
w_7=(\Omega_7,\Omega_\infty)
(\Omega_\infty,\Omega_3)
(\Omega_3,\Omega_\infty).
\]
Cancel the detour pair,
\[
(\Omega_7,\Omega_\infty)(\Omega_\infty,\Omega_3)
\longrightarrow 0,
\]
then reduce to obtain
\[
w_7'=(\Omega_7,\Omega_\infty),
\qquad
n_7=1=C(\Omega_7).
\]

\paragraph{Cable $\pi_8$.}
\[
w_8=(\Omega_8,\Omega_3)
(\Omega_3,\Omega_\infty)
(\Omega_\infty,\Omega_3)
(\Omega_3,\Omega_\infty).
\]
Cancel
\[
(\Omega_3,\Omega_\infty)(\Omega_\infty,\Omega_3)
\longrightarrow 0,
\]
then apply the transitive rule to obtain
\[
w_8'=2(\Omega_8,\Omega_\infty),
\qquad
n_8=2=C(\Omega_8).
\]

The coefficients are
\[
(n_1,\dots,n_8)=(1,2,1,3,2,2,1,2).
\]
Hence, the reduced cable-word is
\[
\mathcal{W}
=(1,\Omega_1)(2,\Omega_2)(1,\Omega_3)
(3,\Omega_4)(2,\Omega_5)
(2,\Omega_6)(1,\Omega_7)(2,\Omega_8).
\]

The degree-weighted volume is
\[
V_{\deg}(\iota)
=\sum_{i=1}^{8}|n_i|\,\mathrm{Vol}(\Omega_i).
\]
This provides a computable lower bound for the swept volume
of any Lipschitz null-homotopy of $\iota(S^2)$.
\end{example}

\paragraph{Linear-Time Complexity}

Assume we use a simple cable system. 
Each cable $\pi_i$
intersects the immersed sphere once whenever it passes
between two regions and does not re-enter a region it
has already left. 
Thus the crossings along each cable
are strictly ordered.

Let $k_i$ be the number of crossings along $\pi_i$ and
let
\[
E=\sum_{i=1}^N k_i
\]
be the total number of symbols in all cable words.
Each cable is processed once. 
The reduction scans the symbols of $w_i$ in order. 
For each new symbol, the algorithm compares it with the immediately preceding symbol and applies one local rule:
\[
(\Omega_a,\Omega_b)(\Omega_b,\Omega_a)\longrightarrow \varnothing,
\]
\[
(\Omega_a,\Omega_b)(\Omega_b,\Omega_c)\longrightarrow
2(\Omega_a,\Omega_c),
\]
and the corresponding rules for opposite orientations.
Each step requires constant time. 
No symbol is revisited.
Hence, reducing $w_i$ takes $O(k_i)$ time. 
Summing over all cables gives
\[
T(E)=\sum_{i=1}^N O(k_i)=O(E).
\]
During the reduction of a cable, the algorithm stores only
the current region and its coefficient. 
Thus, the working memory per cable is $O(1)$. 
After all cables are processed, we store the final indices $\{n_i\}_{i=1}^N$, which requires
$O(N)$ space.

Therefore, the algorithm runs in $O(E)$ time with $O(1)$
working memory per cable and $O(N)$ memory for the output.

\section{Discussion}
The cable-system formulation provides a combinatorial way to compute the degree vector
\[
(C(\Omega_1),\dots,C(\Omega_N))
\]
and the associated degree-weighted lower bound
\[
V_{\deg}(\iota)
=\sum_{i=1}^N |C(\Omega_i)|\,\mathrm{Vol}(\Omega_i).
\]
One may define for each region $\Omega_i$ a time-dependent
index function $F_i(t)$ that decreases in unit steps from
$C(\Omega_i)$ to $0$ as $t$ increases from $0$ to $1$.
If there exists a Lipschitz null-homotopy $H$ whose index
functions follow such a monotone evolution, then 
\[
\mathrm{Vol}(H)=V_{\deg}(\iota),
\]
so $H$ is volume minimizing.

It is natural to ask whether such monotone index functions
can always be constructed for a given immersed sphere.

Another direction is to extend the cable-system framework to
higher dimensions. 
For an immersed $n$-sphere in codimension
one, the complement consists of finitely many $(n+1)$-dimensional
regions, each carrying a Brouwer degree. 
It remains to determine whether a higher-dimensional analogue of the cable construction can recover these degrees combinatorially and produce a degree-weighted lower bound for the $(n+1)$-dimensional swept volume of a null-homotopy.

\printbibliography
\end{document}